\documentclass[reqno]{amsart}

\usepackage{graphicx}
\usepackage{amsthm}
\usepackage{amsmath}
\usepackage{tabularx}
\usepackage{amssymb}
\usepackage{enumerate}
\usepackage{xcolor}

\newcommand{\upd}[1]{\textcolor{black}{#1}}
\newcommand{\updr}[1]{\textcolor{black}{#1}}
\newcommand{\abs}[1]{\left\lvert#1\right\rvert}
\newcommand{\rd}[1]{\left\lfloor#1\right\rceil}
\DeclareMathOperator{\Res}{Res}
\newcommand\TS{\rule{0pt}{3ex}} 
\newcommand\BS{\rule[-1.2ex]{0pt}{0pt}} 
\newtheorem{thm}{Theorem}
\newtheorem{prop}{Proposition}

\begin{document}

\title{The distribution of $k$-free numbers}

\author[M.~J. Mossinghoff]{Michael J. Mossinghoff}\thanks{This work was supported in part by a grant from the Simons Foundation (\#426694 to M.~J. Mossinghoff).}
\address{Center for Communications Research\\
Princeton, NJ, USA}
\email{m.mossinghoff@idaccr.org}

\author[T. Oliveira e Silva]{Tom\'{a}s Oliveira e Silva}
\address{Departamento de Electr{\'o}nica, Telecomunica{\c c}{\~o}es e Inform{\'a}tica / IEETA,
  Universidade de Aveiro, Portugal}
\email{tos@ua.pt}

\author[T.~S. Trudgian]{Timothy S. Trudgian}\thanks{Supported by Australian Research Council Future Fellowship FT160100094.}
\address{School of Science, The University of New South Wales Canberra, Australia}
\email{t.trudgian@adfa.edu.au}

\keywords{Square-free, oscillations, Riemann hypothesis}
\subjclass[2010]{Primary: 11M26, 11N60; Secondary: 11Y35}
\date{\today}

\begin{abstract}
\noindent
Let $R_k(x)$ denote the error incurred by approximating the number of $k$-free integers less than $x$ by $x/\zeta(k)$.
It is well known that $R_k(x)=\Omega(x^{\frac{1}{2k}})$, and widely conjectured that $R_k(x)=O(x^{\frac{1}{2k}+\epsilon})$.
By establishing weak linear independence of some subsets of zeros of the Riemann zeta function, we establish an effective proof of the lower bound, with significantly larger bounds on the constant compared to those obtained in prior work.
For example, we show that $R_k(x)/x^{1/2k} > 3$ infinitely often and that $R_k(x)/x^{1/2k} < -3$ infinitely often, for $k=2$, $3$, $4$, and $5$.
We also investigate $R_2(x)$ and $R_3(x)$ in detail and establish that our bounds far exceed the oscillations exhibited by these functions over a long range: for $0<x\leq10^{18}$ we show that $\abs{R_2(x)} < 1.12543x^{1/4}$ and $\abs{R_3(x)} < 1.27417x^{1/6}$.
We also present some empirical results regarding gaps between square-free numbers and between cube-free numbers.
\end{abstract}

\maketitle

\section{Introduction}\label{secIntroduction}

For an integer $k\geq2$, we say an integer $n$ is \textit{$k$-free} if it has no divisor $d>1$ which is a perfect $k$th power.
Let $Q_k(x)$ denote the number of positive $k$-free integers not exceeding $x$.
It is well known that the $k$-free integers have density $1/\zeta(k)$ in $\mathbb{Z}^+$:
\[
\lim_{x\to\infty} \frac{Q_k(x)}{x} = \frac{1}{\zeta(k)}.
\]
Let $R_k(x)$ denote the error in estimating $Q_k(x)$ by using this density,
\[
R_k(x) = Q_k(x) - \frac{x}{\zeta(k)}.
\]
It is well known that $R_k(x) = O(x^{1/k})$ and $R_k(x) = \Omega(x^{1/2k})$; see the excellent surveys \cite{FGT,Papp} for background and results on $k$-free integers.

Only slight improvements in bounds on $R_k(x)$ have been obtained unconditionally.
The best known upper bound, due to Walfisz in 1963 \cite{Wall}, employed the sharpest known zero-free region of the Riemann zeta function to show that
\[
R_k(x) = O\left(x^{1/k} \exp\left\{-c k^{-8/5} (\log x)^{3/5} (\log\log x)^{-1/5}\right\}\right),
\]
for an absolute positive constant $c$.
Significantly better bounds are known under the assumption of the Riemann hypothesis.
In 1981, Montgomery and Vaughan \cite{Monty} proved with this that for any $\epsilon>0$ one has
\[
R_k(x) = O\left(x^{1/(k+1)+\epsilon}\right).
\]
This bound has seen improvements for each $k$, again under the assumption of the Riemann hypothesis.
For $k=2$, so concerning the distribution of square-free integers, the best result\footnote{We note that Cohen et al.\ \cite[p.\ 54]{Cohen} claim that the Riemann hypothesis implies $R_2(x) = O(x^{1/4} \log x)$, but this appears to be in error. Indeed, in their Lemma~1, the best choice of $n$ appears to give nothing better than $R_2(x) = O(x^{1/3 + \epsilon})$, which is the result of Montgomery and Vaughan \cite{Monty}.} is due to Liu \cite{Liu}, who proved in 2016 that $R_2(x) = O(x^{11/35 +\epsilon})$.
\upd{For $k=3$ and $4$, in 2010 Baker and Powell \cite{BakerPowell} established that $R_k(x) = O(x^{\theta_k+\epsilon})$, with $\theta_3=17/74$ and $\theta_4=17/94$, and in 2014 Liu \cite{hungry} derived improved results for larger $k$: $\theta_5=23/154$, $\theta_6=4/31$, and $\theta_k=1/(k+\frac{18}{11})$ for $k\geq7$.
In addition, improved values for $7\leq k\leq 20$ were given by Graham and Pintz \cite{GrahamPintz}.}

In the other direction, Evelyn and Linfoot \cite{EL} in 1931 established the unconditional lower bound that $R_k(x) = \Omega(x^{1/2k})$, though without an explicit constant.
Effective lower bounds of this shape were first established by Stark \cite{Stark} in 1966, in the course of studying the Schnirelmann density of the $k$-free integers (i.e., $\inf_{n\geq1} Q_k(n)/n$) and showing that for each $k$ this quantity is strictly smaller than the asymptotic density $1/\zeta(k)$.
Stark proved that
\begin{equation}\label{eqnStark}
\liminf_{x\to\infty} \frac{R_k(x)}{x^{1/2k}} \leq -B_k \quad\mathrm{and}\quad
\limsup_{x\to\infty} \frac{R_k(x)}{x^{1/2k}} \geq B_k,
\end{equation}
where
\[
B_k = 2\left(1-\frac{\gamma_1}{\gamma_2}\right)\abs{\frac{\zeta(\rho_1/k)}{\rho_1\zeta'(\rho_1)}}
\]
and $\rho_j=\frac{1}{2} + i\gamma_j$ for $j=1$, $2$ denote the first two zeros of the Riemann zeta function on the critical line in the upper half plane.
This produces the bounds $B_2=0.0657$, $B_3=0.0388$, $B_4=0.0297$, and $B_5=0.0261$ (truncating after four decimal places).
Later, Balasubramanian and Ramachandra \cite{BR1} (see also \cite{BR2}) employed different analytic methods to obtain some explicit bounds, showing that
\[
\frac{R_2(x)}{x^{1/4}} > 10^{-1000}
\]
infinitely often, and similarly that this quantity is less than $-10^{-1000}$ infinitely often.
Also, in 1983 Pintz \cite{Pintz} investigated the mean value of $\abs{R_k(x)}$ over the interval $[1,y]$, and showed that this exceeds  $c(k) y^{1/2k}$  when $y$ is sufficiently large, where $c(k)$ is an effectively computable constant depending on $k$.

If in fact $R_{k}(x) = O(x^{1/2k})$, then it is straightforward to show that the Riemann hypothesis follows, as well as the simplicity of the zeros of the zeta function on the critical line.
Using a result of Landau, these results follow in fact from the weaker result that $R_k(x)/x^{1/2k}$ is bounded by a constant either from above or from below (see \cite{Ingham1942,MossT1}).
However, following Ingham \cite{Ingham1942}, more is true, even under this weaker hypothesis.
If $R_k(x)/x^{1/2k}$ is bounded in either direction, then it would follow that there exist infinitely many integer relations among the ordinates of the zeros of the Riemann zeta function on the critical line in the upper half plane.
Since there seems no particular \textit{a priori} reason why such linear relations should exist (see, e.g., \cite{Best}), one may suspect that $R_k(x)/x^{1/2k}$ exhibits unbounded oscillations.

In this article, we establish that $R_k(x)/x^{1/2k}$  exhibits considerably larger oscillations than those established in prior work.
Our method produces a nontrivial value in place of the constant $B_k$ in~\eqref{eqnStark} for any $k$, and for $k\leq5$, our bound is particularly simple.
We record these values in the following statement, along with an asymptotic result.
Information on bounds for additional values of $k$ may be found in Section~\ref{secOscillations} (see Figure~\ref{figOscBound100}).

\begin{thm}\label{thmOscillate}
For $2\leq k\leq5$, we have
\begin{equation}\label{board}
\liminf_{x\to\infty} \frac{R_k(x)}{x^{1/2k}} < -3 \quad\mathrm{and}\quad
\limsup_{x\to\infty} \frac{R_k(x)}{x^{1/2k}} > 3.
\end{equation}
In addition,
\[
\liminf_{x\to\infty} \frac{R_k(x)}{x^{1/2k}} < -0.74969 \quad\mathrm{and}\quad
\limsup_{x\to\infty} \frac{R_k(x)}{x^{1/2k}} > 0.74969
\]
for sufficiently large $k$.
\end{thm}

One expects however that very large $x$  are required for $R_k(x)/x^{1/2k}$ to exhibit large oscillations.
Recently Meng \cite{Meng}, building on work of Ng \cite{Ng}, conjectured that for each $k\geq2$ there exists a constant $\beta_k>0$ such that
\[
\limsup_{x\rightarrow\infty} \frac{R_k(x)}{x^{\frac{1}{2k}} (\log\log x)^{\frac{k-1}{2k}} (\log\log\log x)^{\frac{1}{4k}}} = \beta_k,
\]
and
\[
 \liminf_{x\rightarrow\infty} \frac{R_k(x)}{x^{\frac{1}{2k}} (\log\log x)^{\frac{k-1}{2k}} (\log\log\log x)^{\frac{1}{4k}}} = -\beta_k.
\]

Here, we investigate the distribution of square-free and cube-free integers in further detail, and find that $R_2(x)/x^{1/4}$ and $R_3(x)/x^{1/6}$ do not exhibit oscillations nearly as large as those guaranteed by Theorem~\ref{thmOscillate} over a very large interval.

\begin{thm}\label{statue}
If $0< x \leq 10^{18}$ then $-1.12543 x^{1/4} < R_2(x) < 1.11653 x^{1/4}$.
\end{thm}

\begin{thm}\label{statue2}
If $0< x \leq 10^{18}$ then $-1.13952 x^{1/6} < R_3(x) < 1.27417 x^{1/6}$.
\end{thm}

\upd{Theorem \ref{statue} improves on the range $0< x \leq 3.5\cdot 10^{14}$ computed by Kotnik and van de Lune (see \cite{Aria}), who reported that $\abs{R_{2}(x)} < 1.126 x^{1/4}$ in this range.
(A location near $x=1.54\cdot 10^{14}$ that produces a value close to $1.126$ is reported in Section~\ref{secSquarefree}.)}

This article is organized in the following way.
Section~\ref{secWeakInd} reviews the analytic tools required in our method, and describes the notion of weak independence of a set of real numbers and its relationship to oscillation problems in number theory.
Section~\ref{secOscillations} details the computations that allow us to prove Theorem~\ref{thmOscillate}.
Section~\ref{secSquarefree} presents calculations on $R_2(x)$ and related quantities and establishes Theorem~\ref{statue}.
Section~\ref{secCubefree} describes calculations on $R_3(x)$ and related items and verifies Theorem~\ref{statue2}.

\section{Weak independence and oscillations}\label{secWeakInd}

Let $\mu^{(k)}(n)$ denote the characteristic function for the $k$-free integers, so that $\mu^{(2)}(n)=\mu^2(n)$, where $\mu(n)$ denotes the usual M\"obius function.
Writing $s=\sigma+it$ with $\sigma$ and $t$ real, for $\sigma>1$ we have
\begin{equation*}\label{plane}
\frac{\zeta(s)}{\zeta(ks)} = \sum_{n=1}^{\infty} \frac{\mu^{(k)}(n)}{n^{s}} = s \int_{1}^{\infty} \frac{Q_k(x)}{x^{s+1}}\, dx = \frac{s}{\zeta(k)(s-1)} + s\int_{1}^{\infty} \frac{R_k(x)}{x^{s+1}}\, dx.
\end{equation*}
Recasting this as
\begin{equation}\label{train}
\frac{\zeta(s)}{\zeta(ks)} - \frac{s}{\zeta(k)(s-1)} = s\int_{1}^{\infty} \frac{R_k(x)}{x^{s+1}}\, dx,
\end{equation}
we see that if $R_k(x) = O(x^{1/2k})$, then the right side of~(\ref{train}) is analytic for $\sigma > 1/2k$, hence so too is the left side.
\upd{(Note the left side is analytic at $s=1$ since the residue of $\zeta(s)$ at $s=1$ is $1$.)
We should like to conclude from this that there are no zeros of $\zeta(ks)$ for $\sigma > 1/2k$, which implies the Riemann hypothesis.
We shall do this following an argument thoughtfully provided to us by Keith Conrad.
Suppose that the Riemann hypothesis were false, so that $\zeta(s_{0}) = 0$ for some $s_{0} = \sigma_{0} + it_{0}$, with $1/2 < \sigma_{0} < 1$ and $t_{0}>0$.
Then $\zeta(ks)$ has a zero at $s= s_{0}/k$.
Given that the left side of \eqref{train} is analytic for $\sigma>1/2k$, we must have that $\zeta(s_{0}/k) = 0$ as well. Hence, by the functional equation, $\zeta(s_1) = 0$ with $s_1 = 1- s_{0}/k$.
Note that $\Re(s_1) = 1- \sigma_{0}/k > 1/2$ and $\Im(s_1) = t_{0}/k$. 
We iterate this process, manufacturing a sequence $\{s_i\}$ with $\zeta(s_i) = 0$, $\Re(s_i) > 1/2$, and $\Im(s_i)=t_0/k^i$.
This produces a nontrivial zero of $\zeta(s)$ lying very close to the real axis, a contradiction.}

By following an argument \upd{described for example} by Ingham \cite{Ingham1942} (see also \cite{MossT1,MossT2}), one can also show that if $R_k(x) < C x^{1/2k}$ or if $R_k(x) > -C x^{1/2k}$ for a positive constant $C$ and all sufficiently large $x$, \upd{then the Riemann hypothesis follows} and the zeros of the Riemann zeta function are simple.
Thus, if the Riemann hypothesis is false, or if the zeta function has a zero on the critical line with multiplicity greater than $1$, then $R_k(x)/x^{1/2k}$ must exhibit unbounded oscillations in both the positive and negative directions.
In what follows we therefore assume the Riemann hypothesis and the simplicity of the zeros.

Let
\begin{equation}\label{car}
F_k(s) = \frac{\zeta(s)}{s\zeta(ks)} - \frac{1}{\zeta(k)(s-1)}.
\end{equation}
It follows that 
\begin{equation}\label{bike}
G_k(s) := F_k\left(s+ \frac{1}{2k}\right) = \int_{0}^{\infty} R_k(e^{u}) e^{-u/2k} e^{-su}\, du.
\end{equation}
Letting $A_k(u) = R_k(e^{u})e^{-u/2k}$, we see by~(\ref{bike}) that $G_k(s)$ is the Laplace transform of $A_k(u)$.
By assumption, the function $G_k(s)$ has a simple pole at $s=i\gamma/k$ for each $\gamma$ corresponding to a simple zero of the zeta function at $\rho=\frac{1}{2}+i\gamma$, but is otherwise analytic for $\sigma\geq 0$.
Let $T>1$ be a real number, and let $m=m(T)$ denote the number of zeros $\rho=\frac{1}{2}+i\gamma$ of $\zeta(s)$ with $0<\gamma<T$.
Let
\[
G_k^*(s) = \sum_{\abs{n}=1}^m \frac{\Res(G_k,i\gamma_n/k)}{s-i\gamma_n/k} = \sum_{\abs{n}=1}^m \frac{\zeta(\rho_n/k)}{\rho_n\zeta'(\rho_n)(s-i\gamma_n/k)},
\]
so that $G_k(s)-G_k^*(s)$ is analytic in the region $\sigma\geq0$, $\abs{t}\leq T/k$, and let $B_k^*(u)=B_{k,T}^*(u)$ denote the inverse Laplace transform of $G_k^*(s)$, with an admissible weight function attached:
\begin{equation}\label{eqnBk}
B_k^*(u) = 2\Re \sum_{0<\gamma_n<T} \frac{\zeta(\rho_n/k)}{\rho_n\zeta'(\rho_n)} \kappa_{T/k}(\gamma_n/k) e^{i\gamma_n u/k}.
\end{equation}
We say a weight function $\kappa_T(t)$ is \textit{admissible} if it is nonnegative, even, supported on $[-T,T]$, takes the value $1$ at $t=0$, and is the Fourier transform of a nonnegative function.
A result of Ingham \cite{Ingham1942} (see also \cite{MossT1}) then implies that
\begin{equation*}\label{eqnChain}
\liminf_{u\to\infty} A_k(u) \leq \liminf_{u\to\infty} B_k^*(u) \leq B_k^*(v) \leq \limsup_{u\to\infty} B_k^*(u) \leq \limsup_{u\to\infty} A_k(u),
\end{equation*}
for any positive real number $v$.
Thus, one may establish large oscillations in $A_k(u)$ by showing that $B_k^*(u)$ attains large values in the positive and negative directions, or indeed by computing particular values for $B_k^*(v)$.

Using Perron's formula \cite[Lem.\ 3.12]{Touchy} on the Dirichlet series $\zeta(s)/\zeta(ks) = \sum_{n=1}^\infty \mu^{(k)}(n)/n^s$, one finds that
\[
\frac{R_k(e^u)-1}{e^{u/2k}} = \lim_{T\to\infty} 2\Re \biggl(\sum_{0<\gamma_n<T} \frac{\zeta(\rho_n/k)}{\rho_n\zeta'(\rho_n)} e^{i\gamma_n u/k}\biggr) - \sum_{j=1}^\infty \frac{\zeta(-2j/k)}{2j\zeta'(-2j)}e^{-u(4j+1)/2k}.
\]
Since the latter sum decays rapidly in $u$, one might expect $B_k^*(u)$ from \eqref{eqnBk} to approximate our normalized function $R_k(e^u)e^{-u/2k}$ when $T$ is large, and when one chooses the (inadmissible) weight function $\kappa_T(t)=1$.
Figure~\ref{figBkstar} displays plots of $B_k^*(\log x)$ with $\kappa_T(t)=1$ for $k=2$, $3$, $4$, and $5$, using the first $2000$ zeros of the zeta function in the upper half plane (so $T=2516$).
Here, $200$ equispaced values per decade were used to plot the functions.
One may compare part~(a) in this figure with the actual plot of minimum and maximum values of $R_2(x)/x^{1/4}$ over a similar interval in Figure~\ref{f:true_error}.
Note for instance the large negative spike here just past $x=10^{14}$, which plays a role in Theorem~\ref{statue}.
Similarly, one may compare Figure~\ref{figBkstar}(b) with Figure~\ref{f:true_cube_error}.

\begin{figure}[tb]
  \begin{center}
    \includegraphics{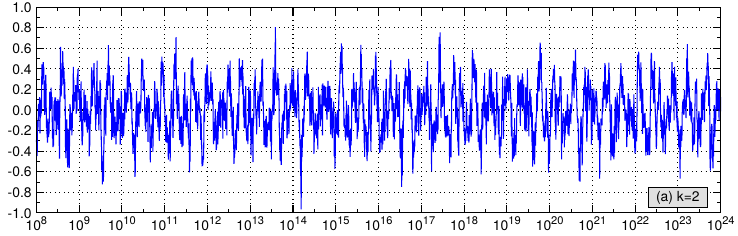} \\
    \includegraphics{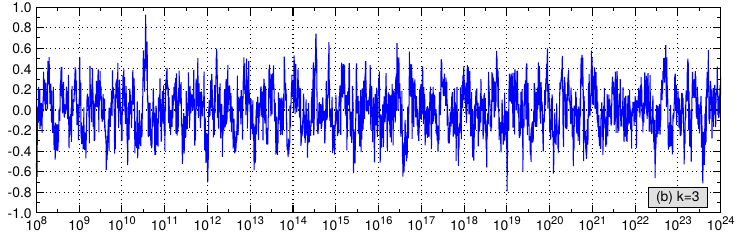} \\
    \includegraphics{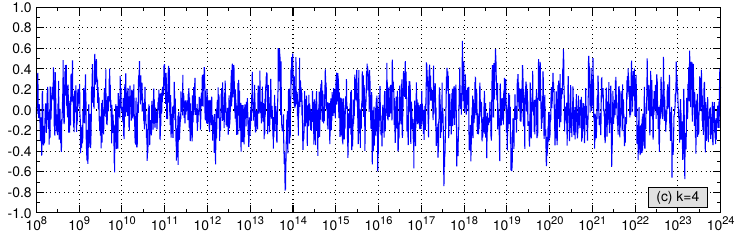} \\
    \includegraphics{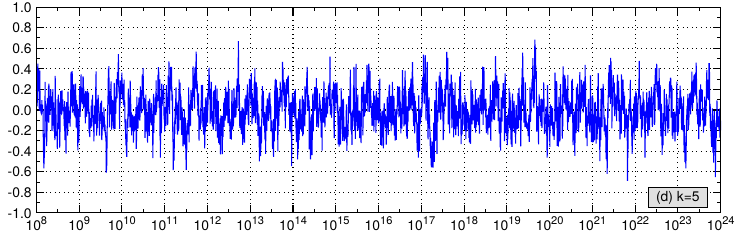} \\
  \end{center}
  \caption{$B_k^*(\log x)$ from~\eqref{eqnBk}, using $\kappa_T(u)=1$, and $T=2516$, so incorporating $m=2000$ zeros of $\zeta(s)$, over $10^8\leq x\leq10^{24}$.}\label{figBkstar}
\end{figure}

Again following Ingham (see also \cite[Lem.~3]{Meng}), we find that large oscillations in $B_k^*(u)$ would follow if one could establish that the $\{\gamma_n\}$ are linearly independent over $\mathbb{Q}$.
Using simultaneous approximation, for fixed $k$ one may choose sequences $\{v_j\}$ and $\{w_j\}$, with $v_j\to\infty$ and $w_j\to\infty$, so that  for each $n\leq m$ we have
\[
\arg\left(\frac{\zeta(\rho_n/k)}{\rho_n\zeta'(\rho_n)}e^{i\gamma_n v_j/k}\right)<\epsilon
\]
for each $j$. Similarly for the $w_j$ we can require the argument of each term to be very near $\pi$.
In addition, under the assumption of the Riemann hypothesis and the simplicity of the zeros of the zeta function, using a method similar to that employed in \cite[Section 14.27]{Touchy} for $\sum_n 1/\abs{\rho_n\zeta'(\rho_n)}$, one may show that\footnote{Since $\Re(\rho_{n}/k) < 1/2$ we need to adapt slightly the method given by Ingham \cite[pp.\ 317--318]{Ingham1942}: choosing $1/2<\alpha<\beta<1$ in Ingham's proof suffices.}
\[
\sum_{n=1}^\infty \abs{\frac{\zeta(\rho_n/k)}{\rho_n\zeta'(\rho_n)}} = \infty.
\]
 We therefore conclude that linear independence of the ordinates of the zeros of $\zeta(s)$ in the upper half plane, or even if only finitely many such linear relations exist, implies that $R_k(x)/x^{1/2k}$ must exhibit unbounded oscillations.
Thus, $R_k(x) = O(x^{1/2k})$ would imply that there are infinitely many such dependencies.
One therefore suspects, by analogy with similar work on the Mertens function \cite{Best} and sums involving the Liouville function \cite{MossT2}, that $R_k(x)$ is bounded neither above nor below by a constant multiple of $x^{1/2k}$.
As in those earlier investigations, one may obtain information on the magnitude of the oscillations of this function by drawing on the concept of \textit{weak independence}.
Roughly speaking, while establishing linear independence of the ordinates of the zeros of the zeta function in the upper half plane would establish \textsl{unbounded} oscillations, establishing weak independence would allow one to conclude that \textsl{large} oscillations must exist.

We review the definition of weak independence and its application to oscillation problems. 
Additional details may be found in \cite[\S 5]{MossT2}.
Let $\Gamma$ denote a set of positive numbers, and let $\Gamma'$ denote a subset of $\Gamma\cap[0,T]$, for some real parameter $T$.
We say $\Gamma'$ is \textit{$N$-independent in $\Gamma\cap[0,T]$} if the following conditions both hold.
\begin{enumerate}[(a)]
\item\label{itemWI1} If $\sum_{\gamma\in\Gamma'} c_\gamma \gamma = 0$ with each $c_\gamma\in\mathbb{Z}$ and each $\abs{c_\gamma}\leq N$, then each $c_\gamma=0$.
\item\label{itemWI2} If $\sum_{\gamma\in\Gamma'} c_\gamma \gamma = \gamma^*$ for some $\gamma^*\in\Gamma\cap[0,T]$ with each $c_\gamma\in\mathbb{Z}$ and each $\abs{c_\gamma}\leq N$, then $\gamma^*\in\Gamma'$, $c_{\gamma^*}=1$, and every other $c_\gamma=0$.
\end{enumerate}

Anderson and Stark \cite{AndStark} obtained an analogue of Ingham's result for the setting of weak independence. 
Suppose $g(u)$ is a piecewise-continuous real function which is bounded on finite intervals, and let $G(s)$ denote its Laplace transform,
\[
G(s) = \int_0^\infty g(u)e^{-su}\,du.
\]
Suppose $G(s)$ is absolutely convergent in the half-plane $\sigma>\sigma_0$, and can be analytically continued back to $\sigma\geq0$, except for simple poles on the imaginary axis occurring at $\pm i\gamma$ for $\gamma\in\Gamma$, and possibly at $0$ as well.
Anderson and Stark established the following result in 1981; a proof may also be found in \cite[Thm.~4.1]{MossT2}.

\begin{prop}[Anderson and Stark]\label{propAS}
Let $\Gamma'\subset\Gamma$, and suppose $\Gamma'$ is $N$-independent in $\Gamma\cap[0,T]$.
Then, using the notation above,
\begin{equation*}
\liminf_{u\to\infty} g(u) \leq \Res(G,0) - \frac{2N}{N+1}\sum_{\gamma\in\Gamma'}\kappa_T(\gamma)\abs{\Res(G,i\gamma)}
\end{equation*}
and
\begin{equation*}
\limsup_{u\to\infty} g(u) \geq \Res(G,0) + \frac{2N}{N+1}\sum_{\gamma\in\Gamma'}\kappa_T(\gamma)\abs{\Res(G,i\gamma)},
\end{equation*}
where $\kappa_T(t)$ is an admissible weight function.
\end{prop}

We apply this result to the function $G_k(s)$ from~\eqref{bike}.
While the Fej\'er kernel qualifies as an admissible weight function, we choose the function
\begin{equation}\label{eqnKernel}
\kappa_{T}(t) = \begin{cases}\displaystyle \left(1- \frac{\abs{t}}{T}\right)\cos\left(\frac{\pi t}{T}\right) + \frac{1}{\pi}\sin\left(\frac{\pi\abs{t}}{T}\right), &\quad \abs{t} \leq T, \\
0, &\quad \abs{t}>T,
\end{cases}
\end{equation}
which gives higher weight to values near the origin, at the expense of values farther away.
This kernel was introduced by Jurkat and Peyerimhoff \cite{JP}, and was employed recently by the first and third authors \cite{MossT2} in their study of oscillations in sums involving the Liouville function.

Let $\rho_n=\frac{1}{2}+i\gamma_n$ denote the $n$th zero of the Riemann zeta function on the critical line in the upper half plane, and let $\Gamma=\{\gamma_n\}_{n\geq1}$.
Suppose $\Gamma' = \{\gamma_{i_1}, \ldots, \gamma_{i_n}\}$ with $i_1<\cdots<i_n$ and $\gamma_{i_n}<T$ is $N$-independent in $\Gamma\cap [0,T]$.
Then, combining Proposition~\ref{propAS} with~\eqref{train}, \eqref{car}, and~\eqref{bike}, and observing that $\kappa_{T/k}(t/k)=\kappa_T(t)$ by~\eqref{eqnKernel}, we conclude that
\begin{equation}\label{horse}
\begin{split}
\liminf_{x\rightarrow\infty} \frac{R_k(x)}{x^{1/2k}} \leq -\frac{2N}{N+1} \sum_{j=1}^{n} \kappa_T(\gamma_{i_j}) \abs{\frac{\zeta(\rho_{i_j}/k)}{\rho_{i_j}\zeta'(\rho_{i_j})}},\\
\limsup_{x\rightarrow\infty} \frac{R_k(x)}{x^{1/2k}} \geq \frac{2N}{N+1} \sum_{j=1}^{n} \kappa_T(\gamma_{i_j}) \abs{\frac{\zeta(\rho_{i_j}/k)}{\rho_{i_j}\zeta'(\rho_{i_j})}}.
\end{split}
\end{equation}
In the next section we employ~\eqref{horse} to establish Theorem~\ref{thmOscillate}.

\section{Lower bounds on oscillations}\label{secOscillations}

We describe a method for computing a value for $N$ for which a particular subset of zeros $\Gamma'$ of size $n$ in $\Gamma\cap[0,T]$ is $N$-independent.
This follows the method employed in \cite{MossT2,MossT3}, and additional details may be found there.

Let $m=\abs{\Gamma\cap[0,T]}$; the value of $n\leq m$ will be chosen later.
We create $m-n+1$ matrices $M_0$, \ldots, $M_{m-n}$ in the following way.
Matrix $M_0$ has $n+1$ rows and $n$ columns, and is formed by augmenting the $n\times n$ identity matrix $I_n$ with the additional row $(\rd{2^b\gamma_{i_1}},\ldots,\rd{2^b\gamma_{i_n}})$.
Here, $b$ is a parameter governing the number of bits of precision required, and $\rd{x}$ denotes the integer nearest $x$.
For $1\leq j\leq m-n$, matrix $M_j$ has $n+2$ rows and $n+1$ columns, and is formed in a similar way, by appending a row of the form $(\rd{2^b\gamma_{i_1}},\ldots,\rd{2^b\gamma_{i_n}},\rd{2^b\gamma_j^*})$ to $I_{n+1}$, where $\gamma_j^*$ is selected from $(\Gamma\cap[0,T]) \backslash \Gamma'$.
The columns of matrix $M_j$ generate a $\mathbb{Z}$-lattice $\Lambda_j$ in $\mathbb{R}^n$ (when $j=0$) or in $\mathbb{R}^{n+1}$ (when $j>0$), and one may verify that if $\Gamma'$ is $N$-dependent in $\Gamma\cap[0,T]$, then one of these lattices must contain a relatively short nonzero vector.
For example, if condition~(\ref{itemWI1}) of $N$-independence were violated, so there exists a nontrivial $\mathbf{c}=(c_1,\ldots,c_n)\in\mathbb{Z}^{n}$ with each $\abs{c_j}\leq N$ and $\sum_{j=1}^n c_j \gamma_{i_j} = 0$, then $v=(c_1,\ldots,c_n,\sum_{j=1}^n c_j \rd{2^b\gamma_{i_j}})$ is a nonzero vector in $\Lambda_0$, and one may compute that $\abs{v}^2 \leq (\frac{n^2}{4}+n)N^2$.
A violation of requirement~(\ref{itemWI2}) for $N$-independence likewise implies the existence of a vector in one of the other lattices $\Lambda_j$ with length bounded by a similar expression in $n$ and $N$.
Thus, if we can determine a lower bound on the length of a nonzero vector $v$ in any of the lattices $\Lambda_j$, then we can choose $N$ to guarantee that $\Gamma'$ is $N$-independent in $\Gamma\cap[0,T]$.

We can obtain a lower bound on the length of a nonzero vector in an integer lattice by computing the Gram--Schmidt orthogonalization of a basis of the lattice: no nontrivial vector in the lattice can be smaller than the smallest vector in this orthogonalization (see \cite{MossT2} for a proof).
However, applying Gram--Schmidt to the bases described here would produce very poor results.
Instead, one applies the LLL lattice reduction algorithm \cite{LLL} to the given bases first, in order to produce alternative bases whose component vectors are in a sense more orthogonal, and then we apply Gram--Schmidt to this basis.
This allows us to compute a value $N_j$ for each lattice $\Lambda_j$.
We may then select $N=\min_j \{N_j\}$ and apply Proposition~\ref{propAS}.
We apply this strategy to establish our first theorem.

\begin{proof}[Proof of Theorem~\ref{thmOscillate}]
For $2\leq k\leq5$, we describe values for  $n$, $m$, $b$, and $T$ that allow us to conclude the stated bound of $3$.
We set $T=\gamma_{m+1}-\epsilon$, so that $\gamma_m < T < \gamma_{m+1}$ (here $\epsilon=10^{-10}$ certainly suffices), and some experimentation suggested that $b$ needed to grow linearly in $n$ in order to obtain suitable values for $N$.
This leaves choices for $n$ and $m$.
Once these are selected, we form $\Gamma'$ as the $n$ elements $\gamma\in\Gamma\cap[0,T]$ for which $\kappa_T(\gamma)\abs{\Res(G_k, i\gamma/k)} = \kappa_T(\gamma)\abs{\zeta(\rho/k)/\rho\zeta'(\rho)}$ is largest, where $\rho=\frac{1}{2}+i\gamma$ is a zero of the zeta function with $\gamma\in[0,T]$.
We have some latitude with $m$ and $n$, and therefore choose them to minimize our  computation time.
A large value for $n$ produces many terms in the sum in~\eqref{horse}, so $m$ may not need to be much larger than $n$ in order to achieve a desired result.
While this would mean relatively few matrices to process, the work to process each matrix scales empirically as approximately $n^4$ (and no worse than $O(n^{6+\epsilon})$ from the worst-case run times in LLL, since we choose $b$ linear in $n$).
Thus, computation times grow significantly with $n$, so computational capacity prevents us from picking $n$ too large.
On the other hand, if we restrict $n$ too much, then we need $m$ to be large in order to find enough zeros of the zeta function producing sufficient terms in~\eqref{horse} to achieve our desired bound, so we would have a large number of matrices to process.
This again imposes a computational constraint.
Balancing these considerations leads us to the values chosen here.

For $k=2$, we selected $n=320$, $m=3560$, and $b=9100$, and after applying LLL followed by Gram--Schmidt on these $3241$ matrices, we find the value $N=1630$ suffices.
For $k=3$, we needed only $n=240$, $m=2540$, and $b=6100$, and determine $N=1362$.
For $k=5$, we set $n=290$, $m=2880$, and $b=8000$, and compute that $N=1684$ suffices.
No additional computation was required for $k=4$, since our weak independence result for $k=2$ already sufficed to produce a constant larger than $3$ in~\eqref{board}.
These parameters, along with the value we obtain for $N$ in each calculation, are summarized in Table~\ref{tableWeakIndep}.
This table also lists the indices of the zeros of $\zeta(s)$ utilized in each calculation.

We remark that the LLL algorithm allows specifying a parameter $\delta\in(1/4,1)$, which governs the amount of work performed by this method.
Large values for $\delta$ produce better bases, at a cost of longer run times.
We selected $\delta=0.99$ throughout, as this allowed us to obtain useful results at smaller dimensions. 

Our calculations produce the following more precise values for the lower bound $C_k$ on required oscillations in $R_k(x)/x^{1/2k}$ (truncated at five decimal places):
\begin{equation*}
\begin{split}
C_2 = 3.00119,\; 
C_3 = 3.00096,\; 
C_4 = 3.04262,\; 
C_5 = 3.00187.   
\end{split}
\end{equation*}
For the asymptotic result, we compute the limiting value of the expression on the right in~\eqref{horse} as $k\to\infty$, using the data in the first row of Table~\ref{tableWeakIndep}.
\[
\frac{3260\abs{\zeta(0)}}{1631} \sum_{j=1}^{320} \frac{\kappa_T(\gamma_{i_j})}{\abs{\rho_{i_j}\zeta'(\rho_{i_j})}} = 0.74969468\ldots.\qedhere
\]
\end{proof}

\begin{table}[tb]
\caption{Weak independence calculations for $R_k(x)$.}\label{tableWeakIndep}
\begin{center}
\begin{tabularx}{\textwidth}{|c|c|c|c|c|X|}
\hline
$k$ & $n$ & $m$ & $b$ & $N$ & Indices of zeros of $\zeta(s)$ employed.\\\hline
2 & 320 & 3560 & 9100 & 1630 &
\tiny
1--37, 39--73, 76--83, 85--92, 94--103, 106--125, 128--132, 134--137,
140--147, 150--157, 159, 164--168, 170, 173--175, 177--182, 186, 187
189, 190, 192, 194, 196, 197, 200--205, 210--213, 215--217, 222--228,
233, 234, 237, 238, 240--243, 246--250, 254, 255, 258, 259, 263--266, 271,
273, 274, 276, 280, 286--291, 297--299, 302, 304, 305, 312--316, 323,
327, 328, 330, 331, 336--339, 342, 343, 353--355, 363--366, 368--370, 377--381,
389, 390, 394, 395, 398, 404--406, 419, 421--423, 431, 433, 434, 436, 437,
446--448, 459, 463, 464, 473, 476, 477, 490, 492, 493, 501, 502, 508, 509,
517, 520, 521, 529, 542, 543, 562, 563, 573, 574, 587, 606, 607, 619,
628, 629, 634, 693, 694, 717, 749, 750, 777, 778, 883, 922, 923, 996, 997,
1496, 1497.\\\hline
3 & 240 & 2540 & 6100 & 1362 &
\tiny
1--28, 30--67, 71--79, 81--83, 86--88, 90--94, 97, 98, 102--109, 112--115,
118--127, 133--136, 138--143, 152--158, 162--164, 167, 168, 170--172,
174--176, 185--189, 192, 193, 199, 200, 207--213, 222--225, 233, 234,
240, 242--248, 260, 261, 265--267, 278--283, 298, 299, 304--306, 315, 316,
322, 323, 338--343, 355--358, 363, 364, 368, 369, 378--381, 390, 391,
398, 399, 403, 404, 418--420, 442--445, 458, 463, 480, 481, 485, 486,
508, 509, 525, 526, 547, 548, 566, 579, 591, 606, 607, 717, 830, 996, 997.\\\hline
5 & 290 & 2880 & 8000 & 1684 &
\tiny
1--54, 56--59, 62--73, 75--82, 85--88, 90--101, 107, 111--115, 117--132,
134--136, 142, 144--153, 156, 157, 159, 163--165, 170, 171, 173--175,
178--187, 196--198, 203--210, 212, 213, 215, 216, 221--224, 233--243,
246--248, 258--262, 265--267, 270--272, 274, 278--280, 298--306, 315, 316,
318, 323, 324, 330--332, 335, 336, 338, 339, 342--344, 361--365, 368, 369,
398, 399, 401--404, 424, 429--434, 436, 437, 445, 446, 470--473, 493, 494,
499, 506--509, 534--536, 538, 539, 579, 580, 584, 585, 606, 607, 643, 644,
648, 651, 693, 694, 716--718, 723, 755, 756, 792, 794, 836, 837, 906,
936, 937, 985, 986.\\\hline
\end{tabularx}
\end{center}
\end{table}

The total computation time for this work was approximately $4.7$ core-years, using Intel Xeon Sandy Bridge processors on a large cluster maintained by the National Computational Infrastructure (NCI) in Canberra, Australia.

Last, we remark that the weak independence results summarized in Table~\ref{tableWeakIndep} allow us to deduce lower bounds on the oscillations of $R_k(x)/x^{1/2k}$ for any $k$.
While additional computations of a similar magnitude tuned to a specific value of $k$ would produce somewhat improved values, we may use the existing weak independence results in combination with~\eqref{horse} to obtain a bound for any $k$. 
For example, the $1630$-independence of the $320$ zeros employed in our calculation for $k=2$ produces the bounds $C_6=2.59030$, $C_7=2.43913$, $C_8=2.36244$, $C_9=2.23963$, and $C_{10}=2.11411$.
The values we obtain in this way for $6\leq k\leq100$ are exhibited in Figure~\ref{figOscBound100}.
We remark that $C_k>1$ for $k\leq45$, and $C_{100} = 0.76250$.

\begin{figure}[tb]
  \begin{center}
    \includegraphics{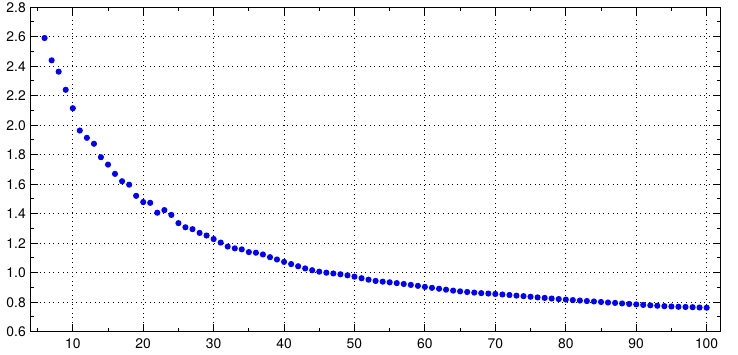}
    \caption{Lower bound $C_k$ on oscillations of $R_k(x)/x^{1/2k}$ for $6\leq k\leq100$.}\label{figOscBound100}
  \end{center}
\end{figure}

\section{Computations on square-free numbers}\label{secSquarefree}

The function $Q_2(x)$ may be computed using the formula~\cite{HW}
\begin{equation*}
  Q_2(x) = \sum_{a=1}^{\lfloor\sqrt x\rfloor} \mu(a)\biggl\lfloor\frac{x}{a^2}\biggr\rfloor,
\end{equation*}
which has a computational complexity of $O(x^{1/2+\epsilon})$.
This can be improved by taking advantage of the fact that $\lfloor x/a^2\rfloor$ is constant over lengthy intervals as $a$ approaches $x$.
Let $x_k=\bigl\lfloor\sqrt{x/k}\bigr\rfloor$, and let $M(x)=\sum_{n\leq x}\mu(n)$ be Mertens' function.
Then it is straightforward to show (see \cite{Pawl} for details) that
\begin{equation}
  Q_2(x) = \sum_{a=1}^{x_n} \mu(a) \biggl\lfloor\frac{x}{a^2}\biggr\rfloor
       + \sum_{k=1}^{n-1} M(x_k) - (n-1)M(x_n),
  \label{e:Qf}
\end{equation}
where $n$ is a positive integer satisfying $n\leq(x/4)^{1/3}$.
In our computations we employed $n=\bigl\lfloor 0.05\sqrt[3] x\bigr\rfloor$.
Even without a fast way to compute $M(x)$, formula (\ref{e:Qf}) can be used to great advantage when preparing a table of values of $Q_2(x)$, because the work of accumulating $\mu(n)$, i.e., of computing $M(x)$ for a sequence of increasing values of $x$, needs to be done only once.
This is especially true if $Q_2(x)$ is computed for many values of $x$ in one run.

In order to have a good idea about how $R_2(x)/x^{1/4}$ behaves for ``small'' values of $x$, and to confirm what the explicit formula for $Q_2(x)$ predicts---cf.\ part~(a) of Figure~\ref{figBkstar}---the values of $Q_2\bigl(\lfloor1/2+10^{k/20000}\rfloor\bigr)$ for all integer values of $k$ satisfying $20000\times 8\leq k\leq 20000\times 28$ were computed using the algorithm described in the previous paragraph.
A geometric progression, with $100$ values per decade, was then used to subdivide the interval $[10^8,10^{28}]$, and the minimum and maximum of $R_2(x)/x^{1/4}$ of the $200$ values belonging to each subinterval were computed.
Figure~\ref{f:error} presents them.
This figure suggests that, up to $10^{28}$, the quantity $\abs{R_2(x)}x^{-1/4}$ stays below~$1.2$; indeed, it appears that it rarely exceeds~$1$.
However, it must be emphasized that the data used to make this figure is sparse, and so the actual lower and upper bounds in each interval will almost certainly differ.
(One may compare this plot with the plot of actual extremal values of $R_2(x)/x^{1/4}$ for $x\leq10^{18}$ in Figure~\ref{f:true_error}.)

\begin{figure}[tbp]
  \includegraphics{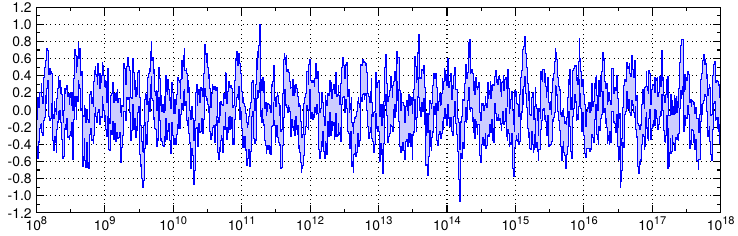}
  \includegraphics{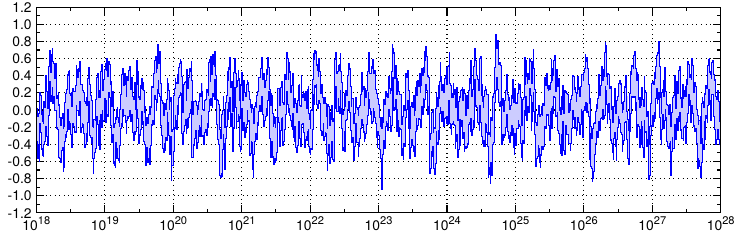}
  \caption{Estimate, based on real but sparse data, of the extremal values of $R_2(x)/x^{1/4}$ over short intervals.}
  \label{f:error}
\end{figure}

On the subintervals mentioned above with upper limits less than or equal to $10^{18}$, these estimates of the minima and maxima were used to speed up the determination of the true minima and maxima of $R_2(x)/x^{1/4}$.
This was done in the following way.
Using a segmented sieve that sieves with the squares of primes,\footnote{To speed things up, the squares of the primes used in the sieve can be subdivided into two classes: those that must have at least one multiple in each segment of the sieve, and those that may have at most one multiple in each segment of the sieve. Those in the first class can be dealt with as in the Eratosthenes sieve. Those in the second class can be dealt with by placing them, and their first multiple not yet taken into consideration, in a priority queue (min-heap); in this way, the multiples of the squares of the primes can be brought into play at the appropriate time in a sufficiently efficient way.} all square-free integers in each interval were identified.
Each segment of the square-free sieve \upd{was} analyzed $64$ integers at a time.
If it was determined that for all possible ways $Q_2(x)$ could increase\footnote{The worst cases are: for the maximum, in an interval of $63$ consecutive integers there can exist $43$ square-free integers, and for the minimum, in an interval of $64$ consecutive integers there can exist no square-free integers.} it was not possible to reach a new minimum or maximum of $R_2(x)/x^{1/4}$, then a fast path in the code was taken; otherwise, all points of increase of $Q_2(x)$ for these $64$ consecutive integers were examined one at a time.
In the fast path, a population count instruction (sideways addition) was used to update the running values of $Q_2(x)$ and a rational approximation to $R_2(x)$.

To avoid a loss of precision when two nearby, large, finite precision real numbers are subtracted, instead of keeping track of the value of $R_2(x)$, our code uses the rational approximation $U_2/V_2$ to $1/\zeta(2)$, with $U_2=4391\,12660\,01254$ and $V_2=7223\,11373\,63897$, and keeps track of the $64$-bit signed integer $V_2Q_2(x)-U_2x$.
For the rational approximation to $1/\zeta(2)$ chosen, for $x\leq 10^{18}$ there is no danger of arithmetic overflow in the value of $V_2Q_2(x)-U_2x$ as long as $\abs{R_2(x)/x^{1/4}}<4$, which is a very pessimistic bound.
Furthermore, up to $10^{18}$, the error of approximating $R_2(x)/x^{1/4}$ by $(V_2Q_2(x)-U_2x)/(V_2x^{1/4})$ is at most $9.7\times 10^{-16}$, and so the computation of $R_2(x)/x^{1/4}$ can be done using standard double-precision floating arithmetic.

To avoid the very small danger of missing a true extremum due to roundoff errors, for each interval we also recorded the $R_2(x)/x^{1/4}$ values that were within $10^{-6}$ of an extremum.
After finishing processing an interval these ``close call'' values (in most cases there were none) were confirmed not to be extremal values.

When the fast path was taken, which was most of the time due to good initial estimates of the minimum and maximum of $R_2(x)/x^{1/4}$ in an interval, processing $64$ consecutive integers required about $48$ clock cycles.
This includes sieving, updating $V_2Q_2(x)-U_2x$, and collecting statistics about gaps between consecutive square-free integers (described below in Section~\ref{fpga}).

\begin{proof}[Proof of Theorem~\ref{statue}.]\label{jim}
  Using the segmented square-free sieve described above, we computed $Q_2(x)$ for all positive integers $k\leq10^{18}$.
  For $k\leq x<k+1$, $Q_2(x)$ is constant and positive.
  In these conditions the first derivative of $R_2(x)/x^{1/4}$ cannot vanish.
  It is therefore enough to examine $R_2(x)/x^{1/4}$ at every point of increase of $Q_2(k)$.
  For minima, the value of $Q_2(x)$ before the point of increase is the one that needs to be taken into consideration.
  For maxima, the value of $Q_2(x)$ after the point of increase is the one that needs to be taken into consideration.

  Up to $10^{18}$ it was found that $R_2(x)/x^{1/4}>-1.12543$ and that $R_2(x)/x^{1/4}<1.11653$.
  The minimum occurred at $x=15495\,33137\,38409-\epsilon$, for which $Q_2(x)=9420\,03189\,39699$ and $R_2(x)/x^{1/4}\approx -1.12542\,91388$.
  The maximum occurred at $x=43$, for which $Q_2(x)=29$ and $R_2(x)/x^{1/4}\approx 1.11652\,25284$.

  Figure~\ref{f:true_error} depicts the actual minima and maxima in all subintervals of $[10^8,10^{18}]$ into which the computation was split.
  The interval $(0,10^8]$ was dealt with separately.
  Table~\ref{t:bananas} presents all relevant data for interesting values of~$x$: minimum smaller than $-1$ or maximum larger than $1$ in a subinterval.

  The computation, which also included the collection of statistics about gaps between consecutive square-free integers, described in Section~\ref{fpga} below, required approximately $6$ core-years.
  Some steps were taken to ensure the correctness of the computation.
  No random errors, due to sporadic hardware faults, were found.
  The computation was double-checked up to \upd{$10^{18}$} using different computers.
\end{proof}

\begin{figure}[tbp]
  \includegraphics{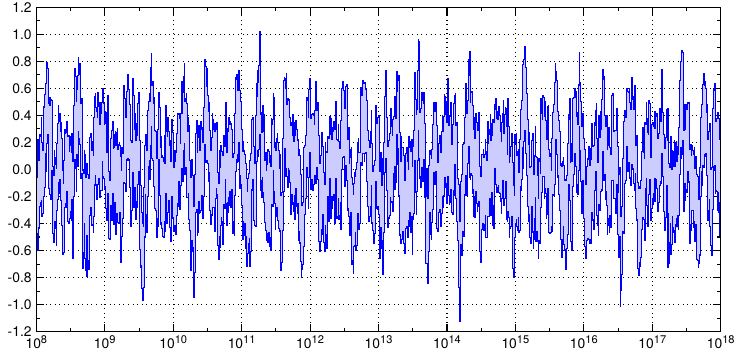}
  \caption{Actual extremal values of $R_2(x)/x^{1/4}$ over short intervals.}
  \label{f:true_error}
\end{figure}

\begin{table}[tbp]
  \centering
  \caption{Some large values of $R_2(x)/{x^{1/4}}$.}
  \label{t:bananas}
  \small
  \begin{tabular}{rrr}\hline
    \noalign{\vspace*{1pt}}
                          $x$ &                  $Q_2(x)$ & $R_2(x-\epsilon)/x^{1/4}$ \\\hline
        $15495\,33137\,38409$ &      $9420\,03189\,39699$ &         $-1.12542\,91388$ \\ 
        $15468\,73074\,99442$ &      $9403\,86065\,38027$ &         $-1.05267\,11482$ \\
    $34\,63868\,12462\,78357$ & $21\,05779\,31020\,81607$ &         $-1.01292\,95585$ \\\hline\hline
    \noalign{\vspace*{1pt}}
                          $x$ &                  $Q_2(x)$ &   $R_2(\epsilon)/x^{1/4}$ \\\hline
                         $43$ &                      $29$ &          $1.11652\,25284$ \\ 
           $18\,31416\,84519$ &        $11\,13367\,94166$ &          $1.02066\,79881$ \\
  \end{tabular}
\end{table}

\subsection{Computations on gaps between square-free numbers}\label{fpga}

If $n$ is a square-free integer and the smallest square-free integer larger than $n$ is $n+g$, then we say $g$ is the \textit{gap} between them.
Because it took little extra time during the computation of $Q_2(x)$ up to $10^{18}$, we also recorded the number occurrences of gaps of different sizes between square-free numbers.
By using one bit per integer to record whether it is square-free, the number of occurrences of a gap size of $1$ within an interval of $64$ consecutive integers can be determined with a logical shift operation, followed by a bitwise \textit{and}, then a population count.
Counting the number of occurrences of larger gaps can be done in essentially the same way, though one must ensure that only gaps between consecutive square-free numbers are counted; this can be done using only two 64-bit registers and elementary bitwise logical operations.
The gap between the last square-free integer in each interval of $64$ integers and the first square-free integer of the next interval has to be handled separately---on contemporary Intel and AMD processors this can be done efficiently by using the bit scan forward and bit scan reverse instructions.

Table~\ref{t:sprained_ankle} reports the number of occurrences for each gap that were observed up to~$10^{18}$.
The location of the first occurrence of each gap was also recorded.
These first occurrences match exactly the known entries of sequence A020754 of the Online Encyclopedia of Integer Sequences.

\begin{table}[tbp]
  \centering
  \caption{Number of occurrences of gaps between square-free numbers.}
  \label{t:sprained_ankle}
  \small
  \begin{tabular}{rrr}
      Gap & Occurrences to $10^{18}$ &          First occurrence \\\hline
      $1$ &  $322\,63409\,89393\,18708$ &                        $1$\\
      $2$ &  $197\,14711\,80333\,78426$ &                        $3$\\ 
      $3$ &   $71\,66013\,71274\,85509$ &                        $7$\\
      $4$ &   $14\,97881\,75413\,00269$ &                       $47$\\
      $5$ &       $93853\,53417\,69417$ &                      $241$\\
      $6$ &       $50066\,36277\,48006$ &                      $843$\\
      $7$ &        $6279\,19138\,50368$ &                    $22019$\\
      $8$ &         $474\,54653\,54639$ &                 $2\,17069$\\
      $9$ &          $12\,56986\,93049$ &                $10\,92746$\\
     $10$ &           $6\,40076\,45662$ &                $88\,70023$\\
     $11$ &              $34483\,07230$ &              $2623\,15466$\\
     $12$ &               $6796\,23488$ &              $2211\,67421$\\
     $13$ &                $146\,82596$ &          $4\,72556\,89914$\\
     $14$ &                $142\,87958$ &          $8\,24625\,76219$\\
     $15$ &                  $5\,57865$ &        $104\,34605\,53363$\\
     $16$ &                     $19206$ &       $7918\,07700\,78547$\\
     $17$ &                       $218$ &   $3\,21522\,63351\,43217$\\
     $18$ &                       $124$ &  $23\,74245\,36409\,00971$\\
     $19$ &                        $11$ & $125\,78100\,08340\,58567$\\
  \end{tabular}
\end{table}

Let $\boldsymbol{h}=\{h_1,\ldots,h_l\}$ be a set of $l$ distinct positive integers. It is known (see, e.g., \cite{Hall,Mirsky}) that
\[
  \sum_{n\leq x} \bigl| \mu(n+h_1)\mu(n+h_2)\cdots\mu(n+h_l) \bigr| =
    A_2(\boldsymbol{h}) x + O(x^{l/(l+1)+\epsilon}),
\]
where
\[
  A_2(\boldsymbol{h}) = \prod_{p}\left( 1-\frac{\nu(p)}{p^2} \right).
\]
Here, $p$ is a prime and $\nu(p)$ denotes the number of distinct residue classes modulo $p^2$ occupied by the offsets $h_i$.
From this result it is possible to compute, using the inclusion-exclusion principle, the theoretical density of a specific gap between consecutive square-free numbers.
In particular, a gap $g$ occurs with density
\[
  D_2(g) = \sum_{\boldsymbol{h}} (-1)^{\abs{\boldsymbol{h}}} A_2(\boldsymbol{h}),
\]
where the sum is over all subsets $\boldsymbol{h}$ of ${ 0,1,\ldots,g }$ that contain both $0$ and $g$.
The quantity $D_2(g)$ was computed by using this method for $1\leq g\leq 56$.
For this range of values of $g$ it suffices to compute, for each subset $\boldsymbol{h}$, the quantity
\[
  C(\boldsymbol{h})=\prod_{2\leq p \leq 7}\bigl(p^2-\nu(p)\bigr),
\]
from which
\[
  A_2(\boldsymbol{h}) = \frac{C(\boldsymbol{h})}{2^2\cdot 3^3\cdot 5^2\cdot 7^2} \prod_{p>7}\left( 1-\frac{\abs{\boldsymbol{h}}}{p^2} \right).
\]
The Euler \upd{products $\prod_{p>7}( 1-k/p^2 )$ were computed for $2\leq k \leq g$} using the \texttt{prodeulerrat} function of the PARI/GP calculator~\cite{pari}.
The computation \upd{of $\sum_{|\boldsymbol{h}|=k} C(\boldsymbol{h})$} was performed and double-checked on four DE2-115 FPGA kits, using a custom-designed digital circuit; \upd{all possible values of $k$ were handled in a single run}.
Each DE2-115 kit, working at $125$ MHz, was about three times faster than an Intel i5-8400T processor (using all six cores turbo-boosted to $3$ GHz).
The computation of $D_2(56)$ took $18$ days.

Table~\ref{t:ice} presents the computed values of $D_2(g)$.
As a check of the correctness of the computations, $\sum_{g=1}^{56} D_2(g)$ was computed using 200 decimal digits and compared to $1/\zeta(2)$; the difference, $1.88013\ldots\times 10^{-72}$ was, as expected, smaller than $D_2(56)$; \upd{$\sum_{g=1}^{56} gD_2(g)$ was also computed and was as close to $1$ as was to be expected (error close to $10^{-70}$).}
The empirical data of Table~\ref{t:sprained_ankle} is in excellent agreement with the theoretical density data.

As depicted in~Figure~\ref{f:unexpected}, we find that $-0.72 g\log g$ is a good approximation to $\log D_2(g)$.
The constant $-0.72$ obtained here differs somewhat from the asymptotic value of $-\frac{6}{\pi^2}(1+o(1))\approx -0.6079(1+o(1))$ obtained in 1997 by Grimmett~\cite[Thm.~1]{Grimmett}.
This may well be due to slow convergence coming from the approximation $p_{k} \sim k \log k$, where $p_{k}$ denotes the $k$th prime, as used in that article.
The next term in this asymptotic expansion is $k\log\log k$, which is not insignificant even for moderate values of $k$, so it is possible that a more refined analysis would better match our results for these small $k$.
We leave this for future research.

\begin{table}[tbp]
  \centering
  \caption{Theoretical densities of gaps between consecutive square-free numbers.}
  \label{t:ice}
  \small
  \begin{tabular}{rlcrl}
     $g$ &                               $D_2(g)$ & $\quad$ &  $g$ &                               $D_2(g)$ \\\hline
    \noalign{\vspace*{2pt}}
     $1$ & $0.32263\,40989\,39245$                &         & $29$ & $0.63177\,19305\,86943\times 10^{-31}$ \\
     $2$ & $0.19714\,71180\,33435$                &         & $30$ & $0.10282\,90469\,64090\times 10^{-30}$ \\
     $3$ & $0.71660\,13712\,76261\times 10^{-1}$  &         & $31$ & $0.17403\,94799\,96015\times 10^{-32}$ \\
     $4$ & $0.14978\,81754\,10999\times 10^{-1}$  &         & $32$ & $0.28084\,80910\,76961\times 10^{-34}$ \\
     $5$ & $0.93853\,53418\,63043\times 10^{-3}$  &         & $33$ & $0.15342\,95147\,58483\times 10^{-36}$ \\
     $6$ & $0.50066\,36278\,21044\times 10^{-3}$  &         & $34$ & $0.11415\,44222\,39253\times 10^{-36}$ \\
     $7$ & $0.62791\,91377\,90450\times 10^{-4}$  &         & $35$ & $0.14860\,11062\,51056\times 10^{-38}$ \\
     $8$ & $0.47454\,65309\,60387\times 10^{-5}$  &         & $36$ & $0.16033\,84614\,66238\times 10^{-40}$ \\
     $9$ & $0.12569\,86997\,40059\times 10^{-6}$  &         & $37$ & $0.50274\,94173\,85908\times 10^{-43}$ \\
    $10$ & $0.64007\,67602\,84955\times 10^{-7}$  &         & $38$ & $0.35373\,30570\,36488\times 10^{-43}$ \\
    $11$ & $0.34482\,93004\,63568\times 10^{-8}$  &         & $39$ & $0.77018\,72886\,84674\times 10^{-44}$ \\
    $12$ & $0.67963\,40254\,86347\times 10^{-9}$  &         & $40$ & $0.12458\,52832\,17433\times 10^{-45}$ \\
    $13$ & $0.14682\,15205\,30483\times 10^{-10}$ &         & $41$ & $0.59476\,17045\,00224\times 10^{-48}$ \\
    $14$ & $0.14286\,40378\,25231\times 10^{-10}$ &         & $42$ & $0.35813\,75704\,70546\times 10^{-48}$ \\
    $15$ & $0.55714\,29760\,16079\times 10^{-12}$ &         & $43$ & $0.40400\,17117\,92499\times 10^{-50}$ \\
    $16$ & $0.19050\,90268\,96753\times 10^{-13}$ &         & $44$ & $0.34084\,06610\,31179\times 10^{-52}$ \\
    $17$ & $0.21748\,39446\,65984\times 10^{-15}$ &         & $45$ & $0.10817\,55378\,21063\times 10^{-54}$ \\
    $18$ & $0.14630\,61929\,51793\times 10^{-15}$ &         & $46$ & $0.63385\,01157\,66665\times 10^{-55}$ \\
    $19$ & $0.39151\,44176\,57893\times 10^{-17}$ &         & $47$ & $0.45781\,30010\,73428\times 10^{-57}$ \\
    $20$ & $0.77538\,08469\,50770\times 10^{-19}$ &         & $48$ & $0.80124\,38339\,00287\times 10^{-58}$ \\
    $21$ & $0.65658\,47892\,78060\times 10^{-21}$ &         & $49$ & $0.28022\,35279\,83020\times 10^{-60}$ \\
    $22$ & $0.61904\,16343\,32994\times 10^{-20}$ &         & $50$ & $0.31333\,60110\,94648\times 10^{-60}$ \\
    $23$ & $0.13737\,01700\,42589\times 10^{-21}$ &         & $51$ & $0.20818\,39530\,55303\times 10^{-62}$ \\
    $24$ & $0.20665\,64577\,26736\times 10^{-23}$ &         & $52$ & $0.45317\,28632\,34738\times 10^{-64}$ \\
    $25$ & $0.12683\,74486\,64448\times 10^{-25}$ &         & $53$ & $0.14537\,32795\,43612\times 10^{-66}$ \\
    $26$ & $0.12367\,12494\,95675\times 10^{-25}$ &         & $54$ & $0.60999\,14000\,07234\times 10^{-66}$ \\
    $27$ & $0.20208\,56647\,02339\times 10^{-27}$ &         & $55$ & $0.38663\,66899\,12914\times 10^{-68}$ \\
    $28$ & $0.87773\,78539\,69717\times 10^{-29}$ &         & $56$ & $0.61221\,28936\,22138\times 10^{-70}$ \\
  \end{tabular}
\end{table}

\begin{figure}[tbp]
  \includegraphics{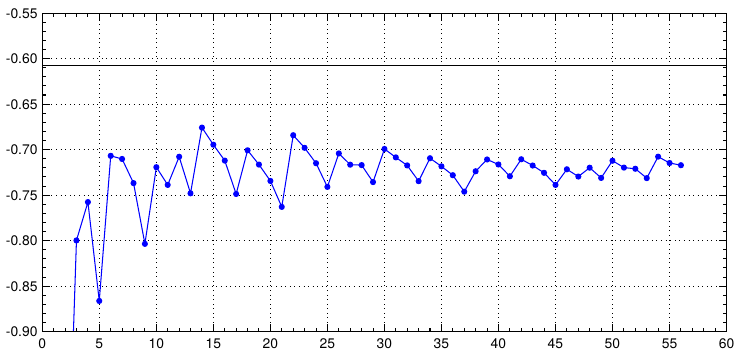}
  \caption{$y=\log D_2(g)/(g\log g)$ for $g\leq 56$, together with the asymptotic value $y=-6/\pi^2$.}
  \label{f:unexpected}
\end{figure}

\updr{%
Let $s_1, s_2, \ldots$ denote the sequence of square-free positive integers.
In 1951, Erd\H{o}s \cite{Erdos51} raised a question concerning moments of the gaps in this sequence, asking for which values of $\gamma\geq0$ does there exist a constant $\beta_2(\gamma)$ for which
\[
S_\gamma(x) := \sum_{s_{n+1}\leq x} (s_{n+1}-s_n)^\gamma \sim \beta_2(\gamma)x.
\]
It is conjectured that such a constant exists for all nonnegative~$\gamma$.
Erd\H{o}s in fact showed this for $\gamma\leq2$, and this has been improved by a number of researchers since then (see \cite[\S6.2]{FGT} for a survey).
Currently the best known result is that such a constant exists for $\gamma \leq 59/16 = 3.6875$, which was proved by Huxley in 2000~\cite{Huxley2000}.
}

\updr{%
If $\beta_2(\gamma)$ exists it must be equal to $\sum_{g=1}^{\infty} g^{\gamma} D_2(g)$.
Thus, the existence of $\beta_2(\gamma)$ is intimately tied to the rate of decay of~$D_2(g)$.
Our $D_2(g)$ data strongly suggests that indeed $\beta_2(\gamma)$ exists for all $\gamma>0$.
}

\updr{%
We may use our data on gaps between square-free numbers to obtain empirical estimates for the values $\beta_2(\gamma)$, for several values of $\gamma$.
We computed $S_\gamma(x)/x$ for $\gamma=2$, $3$, \ldots, $10$ for a number of values of $x\leq 10^{18}$.
In each case, the numerical values obtained appear to converge reasonably quickly.
Figure~\ref{figSFM2} displays the scaled deviations from the limit we calculated for $\gamma=2$ using the values $x=10^{k/100}$ for $800\leq k \leq 1800$.
Table~\ref{tblSFM} exhibits some selected numerical values of $S_\gamma(x)$ for $\gamma\leq7$ and $x\leq10^{18}$.
Additional values for $\gamma\leq10$ and $x=10^k$ with $6\leq k \leq18$ are available in the electronic supplement (see Section~\ref{secAppendix}).
}

\begin{figure}[tbp]
  \includegraphics{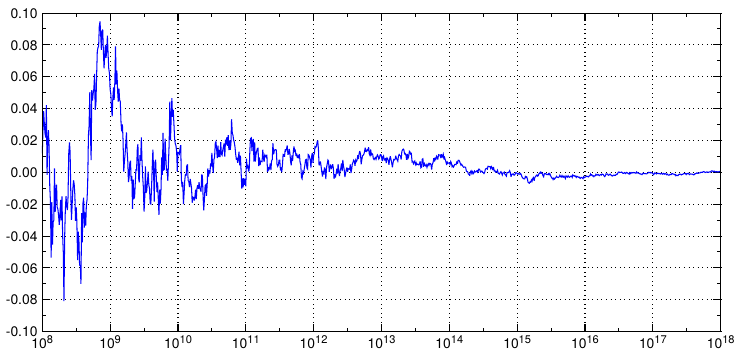}
  \caption{\updr{$\bigl(S_2(x)/x-2.040709776467140\bigr)\sqrt{x}$ for selected values} \updr{of $x$ belonging to the interval $[10^8,10^{18}]$.}}
  \label{figSFM2}
\end{figure}

\begin{table}[tbp]
\tiny
\caption{\updr{Estimates for $\beta_2(\gamma)$ for $\gamma\in\{2,3,4,5,6,7\}$ via computation of $S_\gamma(x)/x$ for several $x$. The last column displays conjectured} \updr{limit values based on the $D_2(g)$ data.}}\label{tblSFM}
\begin{tabular}{|c|rrrrr|r|}\hline
\TS\BS
$\gamma$ & $x=10^6$ & $x=10^9$ & $x=10^{12}$ & $x=10^{15}$ & $x=10^{18}$ & $x=\infty$ \\\hline
\TS
2 &   2.040710 &   2.040711425 &   2.0407097872 &   2.0407097764 &   2.0407097765 &   2.0407097765 \\
3 &   5.042964 &   5.042888653 &   5.0428682597 &   5.0428681130 &   5.0428681138 &   5.0428681138 \\
4 &  14.523182 &  14.523406337 &  14.5232136178 &  14.5232121936 &  14.5232122095 &  14.5232122095 \\
5 &  47.421636 &  47.436267157 &  47.4345636586 &  47.4345517194 &  47.4345519699 &  47.4345519700 \\
6 & 173.110430 & 173.342584985 & 173.3277205474 & 173.3276332827 & 173.3276367967 & 173.3276367966 \\
7 & 701.551764 & 704.297622493 & 704.1675915801 & 704.1671027426 & 704.1671492441 & 704.1671492022 \\\hline
\end{tabular}
\end{table}

\section{Computations on cube-free numbers}\label{secCubefree}

The cube-free case was handled in almost the same way as the square-free case.
Below, we describe only the most relevant differences between the two cases.

The function $Q_3(x)$ can be computed using the formula
\begin{equation*}
  Q_3(x) = \sum_{a=1}^{\lfloor\sqrt[3] x\rfloor} \mu(a)\biggl\lfloor\frac{x}{a^3}\biggr\rfloor.
\end{equation*}
Because evaluating it was fast enough for our purposes, this formula was used to build a table of values of $Q_3(x)$ for $x$ values in a geometric progression up to $10^{28}$, just as was done for the square-free case.
Up to $10^{18}$, these values were used to compute good estimates of the minimum and maximum of $R_3(x)/x^{1/6}$ for each subinterval into which the computation was subdivided.

The constant $1/\zeta(3)$ was approximated by $U_3/V_3$, with $U_3=9372\,16458\,26352$ and $V_3=11265\,87513\,41045$.
The fast path case of the code was taken when it was determined that, based on the current value of $V_3Q_3(x)-U_3x$, no new extrema could occur.
To assess this, the worst cases are: for the maximum, in an interval of $47$ consecutive integers there can exist $42$ cube-free integers, and for the minimum, in an interval of $64$ consecutive integers there can exist no cube-free integers.

\begin{proof}[Proof of Theorem~\ref{statue2}.]
  Using a cube-free segmented sieve, we computed $Q_3(x)$ for all positive integers $k\leq10^{18}$.
  It was found that $R_3(x)/x^{1/6}>-1.13952$ and that $R_3(x)/x^{1/6}<1.27417$.
  The minimum occurred at $x=37\,25506\,68027\,64753-\epsilon$, for which $Q_3(x)=30\,99276\,47392\,06106$ and $R_3(x)/x^{1/6}\approx -1.13951\,13865$.
  The maximum occurred at $x=239$, for which $Q_3(x)=202$ and $R_3(x)/x^{1/6}\approx 1.27416\,63981$.

  Figure~\ref{f:true_cube_error} depicts the actual minimum and maximum in each subinterval of $[10^8,10^{18}]$ into which the computation was split.
  The interval $(0,10^8]$ was dealt with separately.
  Table~\ref{t:oranges} presents all relevant data for interesting values of~$x$: those with minimum smaller than $-1$ or maximum larger than $1$ in a subinterval.

  This computation, which also included the collection of statistics about gaps between consecutive cube-free integers, described below in Section~\ref{last}, required about $4.2$ core-years.
  Some steps were taken to ensure the correctness of the computation.
  No random errors, due to sporadic hardware faults, were found.
  The computation was double-checked up to \upd{$10^{17}$} using different computers.
\end{proof}

\begin{figure}[tbp]
  \includegraphics{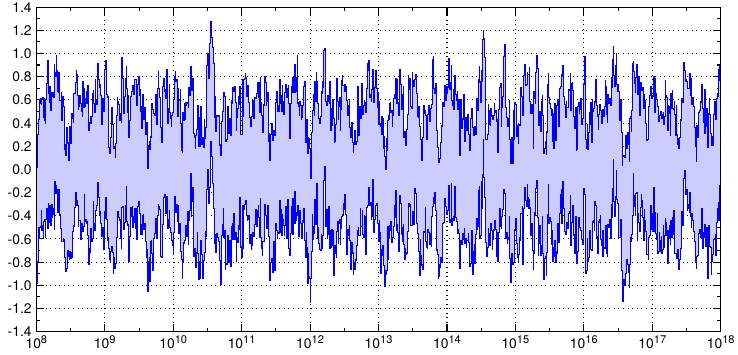}
  \caption{Actual extremal values of $R_3(x)/x^{1/6}$ over short intervals.}
  \label{f:true_cube_error}
\end{figure}

\begin{table}[tbp]
  \centering
  \caption{Some large values of $R_3(x)/{x^{1/6}}$.}
  \label{t:oranges}
  \small
  \begin{tabular}{rrr}\hline
    \noalign{\vspace*{1pt}}
                          $x$ &                  $Q_3(x)$ & $R_3(x-\epsilon)/x^{1/6}$ \\\hline
    $37\,25506\,68027\,64753$ & $30\,99276\,47392\,06106$ & $-1.13951\,13865$ \\
          $100\,67232\,71273$ &        $83\,75005\,11408$ & $-1.13449\,47648$ \\
                 $247\,61514$ &              $205\,99269$ & $-1.05742\,56106$ \\
    $36\,88026\,73969\,44881$ & $30\,68096\,63502\,65741$ & $-1.04604\,23031$ \\
               $43307\,31505$ &            $36027\,67427$ & $-1.03229\,38850$ \\
    $46\,61835\,65920\,74881$ & $38\,78215\,45465\,42909$ & $-1.01702\,12628$ \\
         $1230\,68914\,90897$ &      $1023\,81937\,64677$ & $-1.00027\,80893$ \\\hline\hline
    \noalign{\vspace*{1pt}}
                          $x$ &                  $Q_3(x)$ & $R_3(\epsilon)/x^{1/6}$ \\\hline
                        $239$ &                     $202$ & $1.27416\,63981$ \\
            $3\,58443\,58111$ &         $\,298191\,85851$ & $1.27194\,09798$ \\
            $3\,53865\,20351$ &         $\,294383\,07239$ & $1.20468\,12652$ \\
        $34590\,90473\,40623$ &     $2877\,642867\,25350$ & $1.19557\,11640$ \\
            $3\,63353\,32103$ &         $\,302276\,30727$ & $1.13543\,70876$ \\
        $34787\,05944\,82109$ &     $2893\,961122\,54000$ & $1.13103\,74620$ \\
        $33814\,05939\,98087$ &     $2813\,016531\,16125$ & $1.12409\,97343$ \\
            $3\,41744\,07247$ &         $\,284299\,41404$ & $1.08213\,64423$ \\
        $69306\,16294\,63231$ &     $5765\,630792\,03580$ & $1.07456\,15853$ \\
            $3\,74709\,97111$ &         $\,311723\,98816$ & $1.06156\,70536$ \\
    $26\,92790\,56984\,79471$ & $2\,24015\,232787\,23696$ & $1.05744\,66414$ \\
            $3\,88688\,68031$ &         $\,323352\,97940$ & $1.05056\,75620$ \\
          $164\,33636\,91245$ &       $13\,671263\,70691$ & $1.03873\,29148$ \\
          $160\,58877\,68357$ &       $13\,359498\,74144$ & $1.02234\,21754$ \\
    $29\,93414\,17864\,74199$ & $2\,49024\,332440\,44674$ & $1.00186\,40000$ \\
            $\,317521\,82599$ &         $\,264148\,74856$ & $1.00161\,05574$ \\
  \end{tabular}
\end{table}

\subsection{Computations on gaps between cube-free numbers}\label{last}

Again, the cube-free case was handled in almost the same way as the square-free case.
Table~\ref{t:full_recovery} reports the number of occurrences of each gap that were observed up to~$10^{18}$.
The location of the first occurrence of each gap is also recorded.

\begin{table}[tbp]
  \centering
  \caption{Number of occurrences of gaps between cube-free numbers.}
  \label{t:full_recovery}
  \small
  \begin{tabular}{rrr}
      Gap & Occurrences to $10^{18}$ &         First occurrence \\\hline
      $1$ &  $676\,89273\,70098\,81941$ &                      $1$ \\
      $2$ &  $142\,32586\,49247\,76532$ &                      $7$ \\ 
      $3$ &   $12\,30462\,64590\,31858$ &                     $79$ \\
      $4$ &       $37909\,85862\,31067$ &                   $1374$ \\
      $5$ &         $501\,42960\,01738$ &                  $22623$ \\
      $6$ &           $3\,11951\,93429$ &             $180\,35621$ \\
      $7$ &               $1093\,54086$ &           $43797\,76619$ \\
      $8$ &                  $2\,36827$ &      $120\,42443\,28623$ \\
      $9$ &                       $252$ & $6\,16327\,84732\,58246$ \\
     $10$ &                        $40$ & $2\,60463\,90911\,38247$ \\
  \end{tabular}
\end{table}

In the cube-free case, $A_3(\boldsymbol{h})$ is computed by using
\[
  A_3(\boldsymbol{h}) = \prod_{p}\left( 1-\frac{\nu(p)}{p^3} \right),
\]
where $p$ is a prime and $\nu(p)$ denotes the number of distinct residue classes modulo $p^3$ occupied by the offsets $h_i$.
A gap of size $g$ occurs with density
\[
  D_3(g) = \sum_{\boldsymbol{h}} (-1)^{\abs{\boldsymbol{h}}} A_3(\boldsymbol{h}),
\]
where the sum is over all subsets $\boldsymbol{h}$ of ${ 0,1,\ldots,g }$ that contain both $0$ and $g$.
Table~\ref{t:icecube} presents the \upd{$50$} values of $D_3(g)$ that were computed using this method.
As a check of the correctness of the computations, $\sum_{g=1}^{\upd{50}} D_3(g)$ was computed using 200 decimal digits and compared to $1/\zeta(3)$; \upd{the difference, $4.01206\ldots\times 10^{-167}$ was, as expected, smaller than $D_3(50)$; $\sum_{g=1}^{50}gD_3(g)$ was also computed and was as close to $1$ as was to be expected.}
The empirical data of Table~\ref{t:full_recovery} is in excellent agreement with the theoretical density data.

Over the range $1\leq g\leq 50$, we find that the quantity $\log D_3(g)$ is well approximated by $-1.9g\log g$.
We remark that Grimmett's method \cite{Grimmett} produces the asymptotic formula $\log D_3(g)/(g\log g) = -\frac{2}{\zeta(3)}(1+o(1))$, and $-2/\zeta(3)=-1.6638\ldots$\,.

\begin{table}[tbp]
  \centering
  \caption{Theoretical densities of gaps between consecutive cube-free numbers.}
  \label{t:icecube}
  \small
  \begin{tabular}{rlcrl}
     $g$ &                                $D_3(g)$ & $\quad$ &  $g$ &                                $D_3(g)$ \\\hline
    \noalign{\vspace*{2pt}}
     $1$ & $0.67689\,27370\,09882$                 &         & $26$ & $0.44546\,29821\,75573\times 10^{-68}$  \\
     $2$ & $0.14232\,58649\,24778$                 &         & $27$ & $0.50228\,49540\,01590\times 10^{-72}$  \\
     $3$ & $0.12304\,62645\,90258\times 10^{-1}$   &         & $28$ & $0.35517\,77059\,18650\times 10^{-76}$  \\
     $4$ & $0.37909\,85862\,37504\times 10^{-3}$   &         & $29$ & $0.10948\,56839\,98440\times 10^{-77}$  \\
     $5$ & $0.50142\,95997\,99602\times 10^{-5}$   &         & $30$ & $0.14034\,41927\,96649\times 10^{-81}$  \\
     $6$ & $0.31195\,19555\,11768\times 10^{-7}$   &         & $31$ & $0.10775\,11826\,74150\times 10^{-85}$  \\
     $7$ & $0.10935\,35487\,91799\times 10^{-9}$   &         & $32$ & $0.62150\,75030\,02521\times 10^{-90}$  \\
     $8$ & $0.23670\,56572\,47114\times 10^{-12}$  &         & $33$ & $0.21409\,01531\,52773\times 10^{-94}$  \\
     $9$ & $0.25275\,98567\,71328\times 10^{-15}$  &         & $34$ & $0.38930\,81563\,86705\times 10^{-95}$  \\
    $10$ & $0.42310\,14755\,45621\times 10^{-16}$  &         & $35$ & $0.27672\,90441\,31731\times 10^{-99}$  \\
    $11$ & $0.73494\,59664\,53053\times 10^{-19}$  &         & $36$ & $0.12904\,21686\,35561\times 10^{-103}$ \\
    $12$ & $0.69552\,55994\,15827\times 10^{-22}$  &         & $37$ & $0.44131\,28050\,38986\times 10^{-108}$ \\
    $13$ & $0.43845\,57749\,11792\times 10^{-25}$  &         & $38$ & $0.13260\,60421\,57053\times 10^{-112}$ \\
    $14$ & $0.19972\,82502\,29783\times 10^{-28}$  &         & $39$ & $0.35266\,58540\,53463\times 10^{-117}$ \\
    $15$ & $0.68946\,55450\,60973\times 10^{-32}$  &         & $40$ & $0.83730\,24773\,01521\times 10^{-122}$ \\
    $16$ & $0.18634\,08131\,75014\times 10^{-35}$  &         & \upd{$41$} & \upd{$0.13300\,34889\,45102\times 10^{-126}$} \\
    $17$ & $0.30297\,77572\,65416\times 10^{-39}$  &         & \upd{$42$} & \upd{$0.22933\,45743\,54812\times 10^{-127}$} \\
    $18$ & $0.50536\,07191\,59986\times 10^{-40}$  &         & \upd{$43$} & \upd{$0.79323\,49842\,56904\times 10^{-132}$} \\
    $19$ & $0.15915\,83103\,73791\times 10^{-43}$  &         & \upd{$44$} & \upd{$0.18938\,93144\,19008\times 10^{-136}$} \\
    $20$ & $0.31621\,05214\,23280\times 10^{-47}$  &         & \upd{$45$} & \upd{$0.35690\,93962\,66399\times 10^{-141}$} \\
    $21$ & $0.47462\,27415\,50652\times 10^{-51}$  &         & \upd{$46$} & \upd{$0.59564\,59380\,20661\times 10^{-146}$} \\
    $22$ & $0.57296\,97142\,63535\times 10^{-55}$  &         & \upd{$47$} & \upd{$0.89478\,82503\,39622\times 10^{-151}$} \\
    $23$ & $0.57443\,31927\,85201\times 10^{-59}$  &         & \upd{$48$} & \upd{$0.12228\,99935\,26112\times 10^{-155}$} \\
    $24$ & $0.48797\,53721\,81741\times 10^{-63}$  &         & \upd{$49$} & \upd{$0.11430\,40812\,71746\times 10^{-160}$} \\
    $25$ & $0.26720\,24542\,66525\times 10^{-67}$  &         & \upd{$50$} & \upd{$0.19456\,96735\,87292\times 10^{-161}$} \\
  \end{tabular}
\end{table}

\updr{%
We use these results to obtain empirical information regarding the moments of gaps between cube-free integers, just as we did for the square-free case.
Let $c_1, c_2, \ldots$ denote the sequence of cube-free positive integers.
One may ask for which $\gamma\geq0$ does there exist a constant $\beta_3(\gamma)$ for which
\[
C_\gamma(x) := \sum_{c_{n+1}\leq x} (c_{n+1}-c_n)^\gamma \sim \beta_3(\gamma)x.
\]
The best known result here was established by Huxley in 1996 \cite{Huxley96}, who proved that such a constant exists for $\gamma<11/2$ (and more generally, in the $k$-free case, for $\gamma<2k-1+\frac{2}{k+1}$).
}

\updr{%
Table~\ref{tblCFM} exhibits some selected numerical estimates for $\beta_3(\gamma)$ obtained by computing $C_\gamma(x)/x$ for $\gamma\in\{2,\ldots,7\}$, for a number of values of $x$.
The numerical values appear to converge reasonably well here.
Additional values for $\gamma\leq10$ are available in the electronic supplement.
}

\begin{table}[tbp]
\caption{\updr{Estimates for $\beta_3(\gamma)$ for $\gamma\in\{2,3,4,5,6,7\}$ via computation of $C_\gamma(x)/x$ for several $x$. The last column displays conjectured} \updr{limit values based on the $D_3(g)$ data.}}\label{tblCFM}
\tiny
\begin{tabular}{|c|rrrrr|r|}\hline
\TS\BS
$\gamma$ & $x=10^6$ & $x=10^9$ & $x=10^{12}$ & $x=10^{15}$ & $x=10^{18}$& $x=\infty$ \\\hline
\TS
2 &  1.363095 &  1.363129847 &  1.3631298981 &  1.3631298980 &  1.3631298980 &  1.3631298980 \\
3 &  2.172415 &  2.172620041 &  2.1726204443 &  2.1726204431 &  2.1726204431 &  2.1726204431 \\
4 &  4.049955 &  4.051002839 &  4.0510052027 &  4.0510051847 &  4.0510051846 &  4.0510051846 \\
5 &  8.620327 &  8.625443737 &  8.6254558797 &  8.6254556918 &  8.6254556910 &  8.6254556910 \\
6 & 20.363715 & 20.388367847 & 20.3884269345 & 20.3884253421 & 20.3884253357 & 20.3884253356 \\
7 & 52.297495 & 52.416263161 & 52.4165498874 & 52.4165378845 & 52.4165378390 & 52.4165378387 \\\hline
\end{tabular}
\end{table}

\section{\upd{Appendix}}\label{secAppendix}

\upd{An electronic supplement available with this article contains more extensive data, including values pertaining to Figures~\ref{f:true_error}, \ref{figSFM2}, and~\ref{f:true_cube_error}, as well as Tables~\ref{t:sprained_ankle}, \ref{t:ice}, \ref{tblSFM}, \ref{t:full_recovery}, \ref{t:icecube}, and~\ref{tblCFM}.
For example, the supplement lists estimates for $\beta_2(\gamma)$ and $\beta_3(\gamma)$ for $\gamma\in\{2,\ldots,10\}$ obtained by computing the value of $S_\gamma(x)/x$ and $C_\gamma(x)/x$ using about $1000$ values of $x$ in $[10^8,10^{18}]$, from which one can manufacture plots similar to Figure~\ref{figSFM2} for other $\beta_2(\gamma)$ and $\beta_3(\gamma)$.
Values in the supplement are often listed to higher precision, for example for Tables~\ref{t:ice} and~\ref{t:icecube}.}

\section{Acknowledgments}

The computations described in Section~\ref{secOscillations} were undertaken with the assistance of resources and services from the National Computational Infrastructure (NCI), which is supported by the Australian Government.
Almost all the computations reported in Sections~\ref{secSquarefree} and \ref{secCubefree} were performed in the computer labs of the Electronics, Telecommunication, and Informatics Department of the University of Aveiro.

\upd{We are grateful to Roger Baker, Keith Conrad, Geoffrey Grimmett, H.-Q.\ Liu, and the anonymous referee for providing very valuable feedback.}


\begin{thebibliography}{99}

\bibitem{AndStark}
R.~J. Anderson and H.~M. Stark.
\newblock Oscillation theorems.
\newblock {\em Lecture Notes in Math.}, 899:79--106, 1981.

\bibitem{Aria}
\upd{J. Arias de Reyna and J. van de Lune.}
\newblock \upd{A first encounter with the Riemann hypothesis and its numerical verification.}
\newblock \upd{{\em Gac. R. Soc. Mat. Esp.}, 13(1):109--133, 2010.}


\bibitem{BakerPowell}
R.~C. Baker and K. Powell.
\newblock The distribution of $k$-free numbers.
\newblock {\em Acta Math. Hungar.}, 126:181--197, 2010.

\bibitem{BR1}
R. Balasubramanian and K. Ramachandra.
\newblock Some problems of analytic number theory II.
\newblock {\em Studia Sci. Math. Hungar.}, 14(1--3):193--202, \upd{1979}.

\bibitem{BR2}
R. Balasubramanian and K. Ramachandra.
\newblock On square-free numbers.
\newblock {\em Proceedings of the Ramanujan Centennial International Conference (Annamalainagar, 1987)}, 27--30, RMS Publ., Ramanujan Math. Soc., Annamalainagar, 1988.


\bibitem{Best}
D.~G. Best and T.~S. Trudgian. 
\newblock Linear relations of zeroes of the zeta-function.
\newblock {\em Math. Comp.}, 84(294):2047--2058, 2015.

\bibitem{Cohen}
H. Cohen, F. Dress, and M. El Marraki.
\newblock Explicit estimates for summatory functions linked to the M\"{o}bius $\mu$-function.
\newblock {\em Funct. Approx. Comment. Math.}, 37(1):51--63, 2007.

\bibitem{Erdos51}
\updr{P. Erd\H{o}s.}
\newblock \updr{Some problems and results in elementary number theory.}
\newblock \updr{{\em Publ. Math. Debrecen}, 2:103--109, 1951.}

\bibitem{EL}
C.~J.~A. Evelyn and E.~H. Linfoot.
\newblock On a problem in the additive theory of numbers IV.
\newblock {\em Ann. of Math. (2)}, 32:261--270, 1931.

\bibitem{FGT}
M. Filaseta, S. Graham and O. Trifonov.
\newblock Starting with gaps between $k$-free numbers.
\newblock {\em Int. J. Number Theory}, 11(5):1411--1435, 2015.

\bibitem{GrahamPintz}
S.~W. Graham and J. Pintz.
\newblock The distribution of $r$-free numbers.
\newblock {\em Acta Math. Hungar.}, 53:213--236, 1989.

\bibitem{Grimmett}
G. Grimmett.
\newblock Large deviations in the random sieve.
\newblock {\em Math. Proc. Cambridge Philos. Soc.}, 121:519--530, 1997.

\bibitem{Hall}
R.~R. Hall.
\newblock Squarefree numbers on short intervals.
\newblock {\em Mathematika}, 29:7--17, 1982.

\bibitem{HW}
G.~H. Hardy and E.~M. Wright.
\newblock {\em An Introduction to the Theory of Numbers},
\newblock 5th ed., Oxford Univ. Press, New York, 1979.


\bibitem{Huxley96}
\updr{M.~N. Huxley.}
\newblock \updr{Moments of differences between square-free numbers.}
\newblock \updr{{\em Sieve methods, exponential sums, and their applications in number theory} (Cardiff, 1995), London Math. Soc. Lecture Note Ser., vol. 237, Cambridge Univ. Press, Cambridge, 1997, 187--204.}

\bibitem{Huxley2000}
\updr{M.~N. Huxley.}
\newblock \updr{The rational points close to a curve II.}
\newblock \updr{{\em Acta Arith.}, 93(3):201--219, 2000.}

\bibitem{Ingham1942}
A.~E. Ingham.
\newblock On two conjectures in the theory of numbers.
\newblock {\em Amer. J. Math.}, 64(1):313--319, 1942.

\bibitem{JP}
W. Jurkat and A. Peyerimhoff.
\newblock A constructive approach to Kronecker approximations and its application to the Mertens conjecture.
\newblock {\em J. Reine Angew. Math.}, 286(287):322--340, 1976.


\bibitem{LLL}
A.~K. Lenstra, H.~W. Lenstra, Jr., and L. Lov\'asz.
\newblock Factoring polynomials with rational coefficients.
\newblock {\em Math. Ann.}, 261(4):515--534, 1982.

\bibitem{hungry}
\upd{H.-Q. Liu.}
\newblock \upd{On the distribution of $k$-free integers.}
\newblock \upd{{\em Acta Math. Hungar.}, 144(2):269--284, 2014.}

\bibitem{Liu}
H.-Q. Liu.
\newblock On the distribution of squarefree numbers. 
\newblock {\em J. Number Theory}, 159:202--222, 2016.

\bibitem{Meng}
X. Meng.
\newblock The distribution of $k$-free numbers and the derivative of the Riemann zeta-function.
\newblock {\em Math. Proc. Cambridge Philos. Soc.}, 162(2):293--317, 2017.

\bibitem{Mirsky}
L. Mirsky.
\newblock Arithmetical pattern problems relating to divisibility by $r$th powers.
\newblock {\em Proc. London Math. Soc.}, 50:497--508, 1949.

\bibitem{Monty}
H.~L. Montgomery and R.~C. Vaughan. 
\newblock The distribution of squarefree numbers.
\newblock {\em Recent Progress in Analytic Number Theory}, vol. 1 (Durham, 1979), Academic Press, London, 1981, 247--256.

\bibitem{MossT1}
M.~J. Mossinghoff and T.~S. Trudgian.
\newblock Between the problems of P\'{o}lya and Tur\'{a}n.
\newblock {\em J. Aust. Math. Soc.,} 93(1-2):157--171, 2012.

\bibitem{MossT2}
M.~J. Mossinghoff and T.~S. Trudgian.
\newblock The Liouville function and the Riemann hypothesis.
\newblock {\em Exploring the Riemann Zeta Function: 190 Years from Riemann's Birth}, 201--221, Springer, Cham, 2017. 

\bibitem{MossT3}
M.~J. Mossinghoff and T.~S. Trudgian.
\newblock A tale of two omegas.
\newblock {\em Celebrating 75 years of Mathematics of Computation}, Contemp. Math., \upd{vol. 754,} Amer. Math. Soc., Providence, to appear. 

\bibitem{Ng}
N. Ng.
\newblock The distribution of the summatory function of the M\"{o}bius function.
\newblock {\em Proc. London Math. Soc. (3)}, 89(2):361--389, 2004.

\bibitem{Papp}
F. Pappalardi.
\newblock A survey on $k$-freeness.
\newblock {\em Number Theory}, Ramanujan Math. Sec. Lect. Notes Ser., vol. 1, Ramanujan Math. Soc., Mysore, 2005, 71--88.

\bibitem{pari}
The PARI-Group, PARI/GP version \texttt{2.11.1}, Univ.\ Bordeaux, 2018, available from \texttt{http://pari.u-bordeaux.fr/}.

\bibitem{Pawl}
J. Pawlewicz.
\newblock Counting square-free numbers.
\newblock {\em arXiv:1107.4890v1 [math.NT]}, July 2011.

\bibitem{Pintz}
J. Pintz.
\newblock On the distribution of square-free numbers.
\newblock {\em J. London Math. Soc. (2)}, 28:401--405, 1983.

\bibitem{Stark}
H.~M. Stark.
\newblock On the asymptotic density of the $k$-free integers.
\newblock {\em Proc. Amer. Math. Soc.}, 17(5):1211--1214, 1966.

\bibitem{Touchy}
E.~C. Titchmarsh.
\newblock {\em The Theory of the Riemann Zeta-Function}, 2nd ed., Oxford Univ.\ Press, New York, 1986.


\bibitem{Wall}
A. Walfisz. 
\newblock {\em Weylsche Exponentialsummen in der neueren Zahlentheorie}.
\newblock Mathematische Forschungsberichte, Band XV, Deutscher Verlag der Wissenschaften, Berlin, 1963.

\end{thebibliography}
\end{document}